\documentclass[11pt]{article}
\usepackage{amsmath,amssymb,amsthm,amsfonts,color,longtable}
\usepackage{hyperref}
\usepackage{amssymb,longtable}
\usepackage{amsfonts}
\usepackage[shortlabels]{enumitem}
\DeclareRobustCommand{\SkipTocEntry}[5]{}
\usepackage{setspace}
\usepackage [english]{babel}
\usepackage [autostyle, english = american]{csquotes}
\MakeOuterQuote{"}
\usepackage{amsmath} 
\usepackage{hyperref} 
\usepackage{graphicx}
\usepackage{latexsym}
\usepackage{array}
\usepackage{amssymb}
\usepackage{mathrsfs}

\let\oldbibliography\thebibliography
\renewcommand{\thebibliography}[1]{%
  \oldbibliography{#1}%
  \setlength{\itemsep}{0pt}%
}

\theoremstyle{plain}

\numberwithin{equation}{section}
\newtheorem{thm}{Theorem}

\newtheorem{lemma}[thm]{Lemma}

\newtheorem{defn}[thm]{Definition}

\theoremstyle{remark}

\renewcommand{\div}{{\rm div}}
\newcommand{\rmd}{{\rm d}}
\newcommand{\be}{\begin{equation}}
\newcommand{\ee}{\end{equation}}
\newcommand{\curl}{{\rm curl}} 
\newcommand{\ve}{{\varepsilon}}

\newcommand{\pa}{\partial}
\newcommand{\bq}{\begin{equation}}
\newcommand{\eq}{\end{equation}}

\usepackage[a4paper,
            bindingoffset=0.2in,
            left=1in,
            right=1in,
            top=0.5in,
            bottom=1in,
            footskip=.25in]{geometry}

\title{\vspace{-12mm}On the existence of fibered three-dimensional  perfect fluid equilibria without continuous Euclidean symmetry}
\author{Theodore D. Drivas, Tarek M. Elgindi, and Daniel Ginsberg}
\date{\today}

\begin{document}

\maketitle

\vspace{-8mm}
\begin{abstract}
Following Lortz \cite{L70}, we construct a family of smooth steady states of the ideal, incompressible Euler equation in three dimensions that possess no continuous Euclidean symmetry.  As in Lortz, they do possess a planar reflection symmetry and, as such,  all the orbits of the velocity are closed.  Different from Lortz, our construction has a discrete $m$--fold symmetry and is foliated by invariant  cylindrical  level sets of a non-degenerate Bernoulli pressure, and are flexible (come in infinite dimensional families). Such examples are distinct from but broadly fit in the category of those constructed explicitly by Woolley \cite{W77} and Salat--Kaiser \cite{SK95,KS97}. These narrow the scope of validity of Grad's conjecture on non-existence of fibered equilibria without continuous symmetry. \cite{G1}. 
\end{abstract}

\begin{center}
\textit{Dedicated to Prof. Peter Constantin, our mentor and friend.}
\end{center}
\vspace{-4mm}

\section{Introduction}

Stationary states of the Euler equations play an important role in understanding the fluid motion.   In two dimensions, they are plentiful, and their structure can be rich \cite{DE23}. In three dimensions, far less is understood.  Given a domain $M\subset \mathbb{R}^3$  endowed with the standard Euclidean metric, they are defined by a vector field $u:M\to \mathbb{R}^3$ which is tangent to the boundary of $M$ if non-empty and satisfies
\begin{alignat}{2}
  \label{eul1}
  u\cdot \nabla u + \nabla p &=0, \\
  \label{eul2}
  \div u &=0.
\end{alignat}
In the above, $p:M\to \mathbb{R}$ is the pressure, defined by solving $-\Delta p = \div (u \cdot \nabla u)$ on $M$ with appropriate Neumann boundary data inherited from the equation. Provided the solution is  classical, this can be rephrased in an illuminating way
\be\label{Eulereqn}
u\times {\rm curl} \ \! u = \nabla H, \qquad H := \frac{1}{2} |u|^2+p,
\ee
 where the  function $H:M\to \mathbb{R}$ is called the Bernoulli pressure.  As such, the steady 3D Euler equations also have the interpretation of magnetohydrostatic equilibria \cite{G1, G2, G3, CDG21a}.

 Steady states in 3D can be essentially divided into two categories: those that are nearly two-dimensional in the sense that all orbits are confined to hypersurfaces, and those whose orbits explore volumes. The former are termed \emph{fibered} solutions.  Arnol'd famously classified a subclass of 3D steady states: roughly, those with non-constant Bernoulli pressure are fibered by invariant tori or cylinders \cite{A66, AK09}. See also \cite{S12}. The motion on each torus is quasiperiodic (the field lines are either all closed or all dense). The reason behind this result is that non-constant Bernoulli pressure implies that the velocity and vorticity vector fields are non-vanishing, and transverse at every point. Moreover, the Euler equation \eqref{Eulereqn} implies that $H$ is a first integral of these fields, e.g. $u\cdot \nabla H=0$. Hence, if $\nabla H$ is non-zero, particles are constrained to the level sets of $H$ in a neighborhood of that point. So, on any regular level set  $S:=\{H=c\}$, one finds that $u(x)$ and ${\rm curl} \ \! u(x)$ form a basis for the tangent space for $S$ at each point $x$. The only connected two-dimensional manifolds for which such fields can exist are tori or cylinders.  He  also conjectured that steady states having aligned velocity and vorticity (such as Beltrami flows), and therefore constant Bernoulli function, could behave much more wildly.  
 
 So far, with the notable exceptions of the work of Lortz \cite{L70} (see also \cite{J20}), Woolley \cite{W77} and Salat--Kaiser \cite{SK95,KS97}, all known steady states on $\mathbb{R}^3$ with non-constant Bernoulli pressure have a continuous Euclidean symmetry.  In particular, such fields that are axisymmetric can be build out of a scalar potential solving the  Grad-Shafranov equation.  Outside of Euclidean symmetry, there are very few existence results.  This, together with an understanding of the difficulty for existence, led Grad to conjecture that no fibered smooth solutions, in particular those with non-constant Bernoulli pressure,  exist outside of Euclidean symmetry.  Grad, who was a plasma physicist, was thinking about the confinement fusion problem and was interested in this non-existence result in light of devices such as the tokamak and stellarator, aimed at producing such states for the purposes of confining a nuclear fusion reaction.
 
The only well understood class of steady states outside of symmetry are Beltrami flows, for which there is an easy existence theory. As mentioned above, in general their field lines may be completely chaotic \cite{EP12, EP23}. There is one construction of note due to Lortz, originating from an idea of  Grad and Rubin \cite{G3}, which perturbs a harmonic vector field (and therefore a particular Beltrami flow) with a reflection symmetry which has all of its orbits closed and is therefore fibered by tori  \cite{L70}.  In his construction, the Bernoulli pressure is a given small-amplitude function of the period of revolution of particles on the closed orbits.  Being that the base state is harmonic, in general this period function is general not understood and in simple cases (e.g. the vector field $u_\mathsf{H}=e_z$ on the $\mathbb{R}^2\times \mathbb{T}$), the period function is constant.  As such, it is not clear from Lortz's construction that the pressure levels are cylinders which fiber the domain.  There can be islands separated by regions of compact level sets, of constancy etc.  See also \cite{WS20} for a construction with non-compact pressure surfaces and \cite{SVW} for an extension of Lortz's construction to other systems. Thus, neither Lortz nor Beltrami flows show whether Arnol'd's theorem applies to any steady state having cylindrical levels of the Bernoulli pressure but no continuous symmetry.  On the other hand, the works of Woolley \cite{W77} and Salat--Kaiser \cite{SK95,KS97} produce \emph{explicit} forms of cylindrical solutions which are far from having continuous symmetry (and indeed, are not perturbative of axisymmetric solutions). These solutions illustrate  Arnol'd's theorem applies outside continuous symmetry, but the issue of flexibility and isolation of the explicit objects they find is unclear.

In this work, we demonstrate the existence of such a non-symmetric 3D steady state, which are perturbuations of axisymmetric solutions and come in infinite dimensional families (parametrized by the shape of the outer cylinder).

\begin{thm}
    \label{mainthm}
There exists a cylindrical domain $M\subset \mathbb{R}^2\times \mathbb{T}$, and a $C^\infty$ smooth stationary solution of Euler $u:M\to \mathbb{R}^3$, tangent to $\partial M$,  such that $\nabla H$ is non-vanishing away from a line.  The solution is periodic in the direction set by said line, m-fold symmetric about this line, but possesses no continuous Euclidean symmetry. This solution is a perturbation of a monotone, non-degenerate two-dimensional rotational flow, with the outer boundary of the cylinder a freely prescribed perturbation, up to the discrete symmetries.\end{thm}

Such a solution is fibered by invariant (wobbly) cylinders which are levels of the Bernoulli function.  We remark that the size of the pressure of the solution may be large; see Fig \ref{fig}.

\begin{figure}[h!]
  \begin{center}
  \includegraphics[height=1.5in]{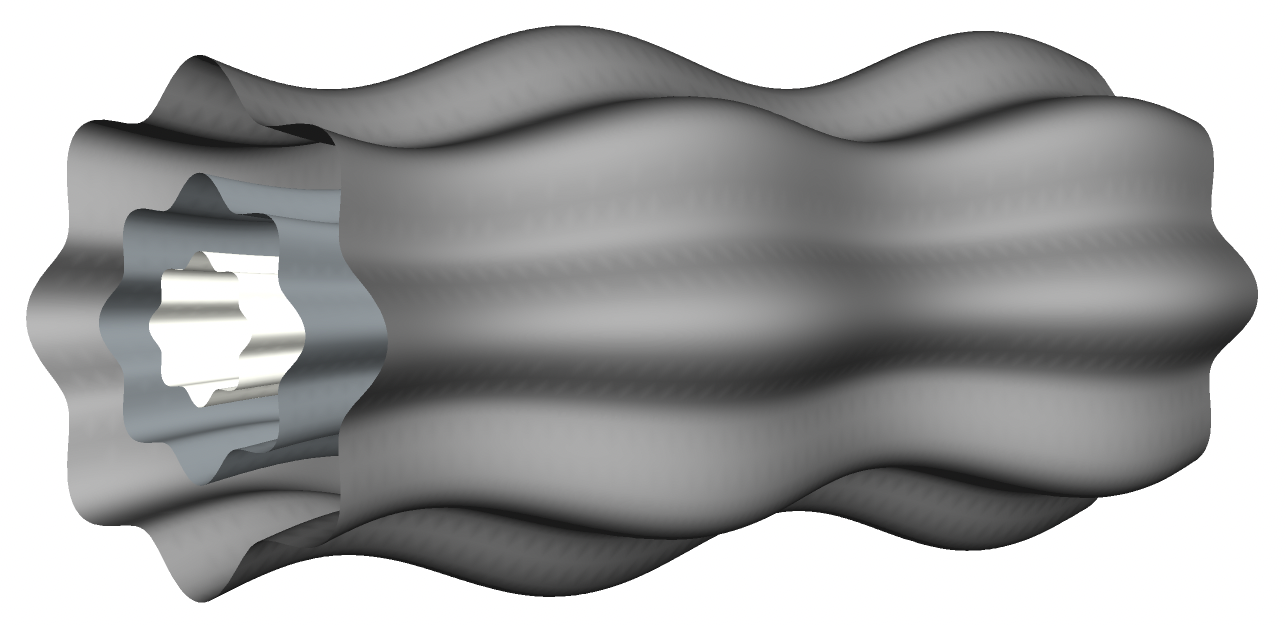}
  \end{center}
  \caption{An $8$--fold symmetric steady state, fibered by levels of the Bernoulli pressure $H$. All orbits of the velocity are confined  to the isosurfaces of $H$, and wrap the "short way".}
  \label{fig}
\end{figure}

In fact, our solution is constructed by modifying Lortz's argument to apply to perturbations of solutions with large Bernoulli pressure (as opposed to harmonic fields, with trivial Bernoulli  pressure).  The reason for Beltrami fields in Lortz's construction is that the pressure provides that small parameter which he uses to break the symmetry and close an iteration.  To overcome this difficulty, our small parameter is instead $1/m$, where $m$ is the multiplicity of discrete reflection symmetries enjoyed by the solution.

A couple concluding remarks should be made.  First, we do not believe the discrete symmetry of the solution is essential, and that there should exist steady state nearby those constructed here which break \emph{all} Euclidean symmetries.  Such an object would certainly be interesting to construct. However, second, it is clear from the discussion of Grad and Rubin \cite{G3} and Grad \cite{G71} that Grad understood that steady states with all closed field lines could exist outside symmetry, although Lortz had the first rigorous such construction \cite{L70}. Our present contribution is to furthermore guarantee the pressure levels are cylinders, allow for large pressure (in the language of confinement fusion, large plasma $\beta$), and are flexible (come in large families of such solutions). The former two properties appear also in the works of Woolley \cite{W77} and Salat--Kaiser \cite{SK95,KS97}. As such, it seems reasonable to stipulate that the Grad's conjecture should include  the provision that the vector field has ergodic orbits  on nearly all its invariant tori.  Such things, Grad clearly states, should be either non-existent or exceptionally rare (isolated) outside of (continuous) Euclidean symmetry \cite{G2}.
Moreover, these hypothetical asymmetric twisting equilibria -- stellarators -- are believed to be relevant objects to advance plasma confinement fusion \cite{L19}.

\section{Steady Euler and Clebsch Variables}

We aim to produce a vector field $u$ which solves steady Euler.  We will do this by perturbing  a given axisymmetric solution $u_*$
which occupies the  "straight" periodic unit cylinder $D_* := \mathbb{D} \times \mathbb{T}$. Specifically, for a smooth $\Omega$ with further conditions to be specified, we will perturb
 \be\label{ustar}
u_*:= \Omega(r)re_\theta, \qquad e_\theta:= \frac{x_1}{r} e_2 - \frac{x_2}{r}e_1
\ee
where $r:= \sqrt{x_1^2+ x_2^2}$, which satisfies \eqref{eul1}-\eqref{eul2} with hydrodynamic pressure
$$
    p_*(r) = \int_0^r \rho \Omega^2(\rho) d\rho.
$$
We note  that Bernoulli pressure $H_* = \frac{1}{2} |u_*|^2 + p_*$ takes the form
$$
    H_*(r) = \frac{1}{2} (r\Omega(r))^2 + \int_0^r \rho \Omega^2(\rho)\, d\rho
    =  \int_0^r\Big[ \rho \Omega'(\rho)+ 2 \Omega(\rho)\Big]\, \rho \Omega(\rho) d\rho
    $$
and so if $\Omega$ is monotone and does not change sign,
$\nabla H_*$ is nonzero except at $\{r=0\}$.

The orbits of $u_*$ are closed and have period $T_*(r) = \frac{2\pi}{\Omega(r)}$. For our construction
we will want to be able to write the Bernoulli function $H_*$ as a smooth function
of $T_*$. For this, we will assume that $\Omega$ satisfies the following two properties:
\begin{itemize}
    \item[] \textbf{(H1)} $r \mapsto \Omega(r)$ is invertible and does not change sign for $r \in [0, \infty)$, and
    \item[] \textbf{(H2)} $|\Omega''(0)| > 0$ 
\end{itemize}
The condition \textbf{(H1)} guarantees that the period $T_*$
is a bounded and invertible function, with inverse $r_*(T) = \Omega^{-1}\left(\tfrac{2\pi}{T}\right)$.
This, together with \textbf{(H2)}, lets us think of the Bernoulli
function $H_*$ as a smooth function of the period $T$. Indeed,
if we define
$
  \mathcal{H}_*(T) = H_*(r_*(T)),
$
and, noting that $r_*'(T) = -\frac{2\pi}{T^2}\frac{1}{\Omega'(r_*(T))}$ we have the formula
\begin{align}
  \mathcal{H}_*'(T) &= -\frac{2\pi}{T^2} \frac{H_*'(r_*(T))}{\Omega'(r_*(T))}
= \frac{2\pi}{T^2} \frac{\Big[ \rho \Omega'(\rho)+ 2 \Omega(\rho)\Big]\, \rho \Omega(\rho) }{\Omega'(\rho)}\Bigg|_{\rho=r_*(T)},
\end{align}
where we used the explicit formula for $H_*$. Noting that smoothness of $u_*$ demands that  $\Omega^{(2n+1)}(0)=0$ for all $n\geq 0$, smoothness of $\mathcal{H}_*$ can then be checked by Taylor expansion of the period function near $r_*=0$.
Repeating the same argument shows higher-order
derivatives of $\mathcal{H}_*$ are also continuous. 

The starting point for the construction is to seek \emph{Clebsch variables}
$H$, $\tau$ associated to $u$,
\begin{align}
  \label{lortz1}
  u\cdot \nabla H &= 0,\\
  \label{lortz2}
  u\cdot \nabla \tau &= 1.
\end{align}
See \cite{C17} for a clear introduction. This is also the starting point of \cite{L70,W77,SK95,KS97}. Given such $H, \tau$, we define
\begin{align}
  \label{lortz3}
  \omega &= \nabla \tau \times \nabla H.
\end{align}
Note that, by construction, we have $\div\, \omega = 0$. Moreover, it holds 
\begin{equation}
  \label{}
  u \times \omega
  = (u\cdot \nabla \tau) \nabla H- (u\cdot \nabla H) \nabla \tau
  = -\nabla H.
\end{equation}
Provided that $u$ satisfies
\begin{align}
  \label{}
  \div \,u&=0,\\
  \curl \,u&= \omega,
\end{align}
the resulting $u$ would satisfy \eqref{Eulereqn} and thus the original
system \eqref{eul1}-\eqref{eul2}. In the rest of this section, we describe
our strategy for solving the above system.

Equations \eqref{lortz1} and \eqref{lortz2} are steady transport equation in three dimensions, and are somewhat delicate because $u$ may have a mix of both closed and ergodic field lines.  Indeed Grad's original objection concerned solving \eqref{lortz2} in the presence of ergodic field lines.  To sidestep this difficultly, we follow Lortz and impose a reflection symmetry so that the constructed  vector field will have all closed orbits. Let us now define the notion of parity:
\begin{defn}
    \label{baseparity}
Fix a plane $\Pi\subset \mathbb{R}^3$. Let $R_\Pi $ be the reflection about $\Pi $.  We say that a function $g:\mathbb{R}^3\to \mathbb{R}$ is odd if $g\circ R_\Pi = - g$ and even if $g\circ R_\Pi = g$.
    We say that a vector field $X:\mathbb{R}^3\to \mathbb{R}^3$ is odd   if $X\circ R_\Pi = -R_\Pi X$.
  We say that $X$ is even     if $X\circ R_\Pi = R_\Pi X$.   
\end{defn}

Note that $X$ is odd  (resp. even) if and only if $R_\Pi^*X = -X$
(resp. $R_\Pi^*X = X)$, where $R_\Pi^*X = R_\Pi^{-1} X\circ R_\Pi= R_\Pi X \circ R_\Pi$ 
denotes the pullback
of $X$ by $R_\Pi$.

    From hereon, without loss of generality we choose $\Pi =\{x_2=0\}$ to be the $x_1-x_3$ plane. In this case, $R_\Pi:=R_ \Pi(x_1, x_2, x_3) = (x_1, -x_2, x_3)$.
Moreover, if $X$ is odd, it means that $X\cdot e_1$ and $X\cdot e_3$ are odd in $x_2$
    and $X\cdot e_2$ is even in $x_2$.   With this definition,
    for any smooth function $\Omega$ the vector field $u_*$ from \eqref{ustar} is
    odd.

We now define the perturbed domain we will construct our solution in.
We consider a perturbed cylindrical domain of this form
\begin{equation}
  \label{dvedef}
  D_\ve = \{r = 1 + \ve g(\theta, z)\}
\end{equation}
where the function $g: \mathbb{S}^1 \times \mathbb{T}$ which is even with respect to 
$\Pi$
and which is $m$-fold symmetric,
\begin{equation}
  \label{}
  g( \theta + 2\pi/m,z) =   g( \theta ,z) 
\end{equation}
for some $m \in \mathbb{Z}_{>0}$.  We aim to construct a solution $u = u_* + \ve u'$ in $D_\ve$
with 
\begin{alignat}{2}
    \label{eul3}
  u\cdot n &= 0, &&\qquad \text{ on } \pa D_\ve,
\end{alignat}
and so that $u$ is odd 
with respect to $\Pi$ (recall Definition \ref{baseparity}).

Once again, the importance of the parity of $u$ is that, together with proximity to
the field $u_*$, it guarantees that all field lines of $u$ are closed
and this will let us find $\tau$ and $H$ satisfying \eqref{lortz1}-\eqref{lortz2}
To see this, we start with the following lemma. 
\begin{lemma}
  \label{closedlines}
  Let $X$ be an odd and Lipschitz vector field on $\mathbb{R}^3$. Let $\gamma$ be any trajectory of $X$
    which passes through the half-planes $\{\theta = 0\}$ and $\{\theta = \pi\}$. Then $\gamma$ is closed.
\end{lemma}
\begin{proof}
      Write $\gamma(t) = r(t)e_r + \theta(t) e_\theta + z(t) e_z$. Without loss
    of generality, we can assume that $\theta(0) = 0$ and $\theta(1) = \pi$.
    We claim that for all $t$,
    \begin{equation}
      \label{claim}
        \theta(1+t) = 2\pi - \theta(1-t),\qquad
        r(1+t) = r(1-t),\qquad
        z(1+t) = z(1-t).
    \end{equation}
    If this holds, then setting $t = 1$ shows that the curve $\gamma$ is closed (with period equal to 2).
    Clearly \eqref{claim} holds at $t = 0$. If we write $\theta_{+}(t) = \theta(1+t)$,
    $\theta_-(t) = 2\pi - \theta(1-t)$ as well as $r_{\pm}(t) = r(1\pm t)$ and
    $z_{\pm}(t) = z(1\pm t)$, then by the parity of $X$, we find
    \begin{equation}
      \label{}
      \frac{\rmd}{\rmd t} r_{\pm} = X_r(r_{\pm}, \theta_{\pm}, z_{\pm}),
      \qquad
      \frac{\rmd}{\rmd t} \theta_{\pm} = X_\theta(r_{\pm}, \theta_{\pm}, z_{\pm})
      \qquad
      \frac{\rmd}{\rmd t} z_{\pm} = X_z(r_{\pm}, \theta_{\pm}, z_{\pm}),
    \end{equation}
    and by the uniqueness theorem for ODE this gives \eqref{claim}.
\end{proof}

In our construction to follow, we will build iteratively a $u$ which perturbs $u_*$:
$u = u_* + \ve  u'$. 
By the discrete symmetry assumptions on $u$, it must vanish along the $z$--axis. 
Using this, since all field lines of $u_*$ pass through both
$\theta = 0$ and $\theta = \pi$,  so do those of $u$ for sufficiently small $\ve$. Thus, by the Lemma, all orbits of such $u$ are closed.

With this knowledge, for such a $u$ we can therefore define a function $\tau$ by
\begin{equation}
  \label{}
    u\cdot \nabla \tau = 1.
\end{equation}
Since $u$ has closed field lines, such $\tau$ is necessarily multi-valued; we orient
the domain so that 
it is well-behaved away from the plane of symmetry $\Pi$ but it jumps across $\Pi\cap \{x_1\geq 0\}$.
Letting $T$ denote the period of the trajectories of $u$, then $u\cdot \nabla T = 0$,
and
if we let
$\tau|_{\Pi_{\pm}}$ denote the limits of $\tau$ taken from
either side of $\Pi$, the period satisfies
\begin{equation}
  \label{Ttau}
  \tau|_{\Pi_+} - \tau|_{\Pi_-} = T.
\end{equation}
For $p \in D_\ve$, $\tau(p)$ is the travel time from $\Pi$ to $p$ along
the orbits of $u$, up to an integer multiple of $T(x)$, where
$x \in \Pi$ is the unique point on $\Pi$ lying on the same orbit
as $p$.

Recalling the period-pressure relation $\mathcal{H}_*(T) = H_*(r_*(T))$, 
we define a Bernoulli function $H$ for $u$ by
\begin{equation}
  \label{qdef}
  H = \mathcal{H}_*(T),
\end{equation}
noting that
$u\cdot \nabla H = \mathcal{H}_*'(T) u\cdot \nabla T = 0$ by construction. We now define
\begin{equation}
  \label{omegadef}
  \omega = \nabla \tau \times \nabla H = \mathcal{H}_*'(T) \nabla \tau \times \nabla T.
\end{equation}
Because
$\tau$ is not continuous, a priori $\omega$ may fail
to be continuous across $\Pi$.
However, since $u$ is odd with respect to $\Pi$, $u|_{\Pi}$ is tangent
to $\Pi$. Thus, because $u \cdot \nabla T = 0$, it follows that
$\nabla T \times \nabla$ involves
only derivatives tangent to $\Pi$, and so the jump in
$\omega$ across $\Pi$ is
\begin{equation}
  \label{}
  \omega|_{\Pi_+} - \omega|_{\Pi_-} 
  = \mathcal{H}_*'(T)  \nabla (\tau|_{\Pi_+} - \tau|_{\Pi_-}) \times \nabla T = 0,
\end{equation}
by \eqref{Ttau}. As a consequence, $\omega$ is continuous across $\Pi$.
We will see later on that in fact $\omega$ is more regular.

To solve the system, it remains to guarantee that with $\omega$
defined as in \eqref{omegadef}, we have $\curl u = \omega$. 
The above motivates the iteration scheme described in the next section.

\section{The Lortz iteration}

We now describe the iteration we will use to construct solutions,
following Lortz \cite{L70}.
Let $u_N$ be an odd and $m$-fold symmetric vector field.
By Lemma \ref{closedlines}, because $u_N$ is odd, provided
$\ve$ is sufficiently small, each integral curve $\gamma_N$ of $u_N$ is closed.
We seek $\tau_N = \tau_* + \tau_N'$ satisfying
\begin{equation}
  \label{tauN}
  u_N\cdot \nabla \tau_N = 1.
\end{equation}
For this, we define $\tau_N'$ by solving
\begin{equation}
  \label{tauNprime}
    u_N\cdot \nabla \tau_N' = (u_* - u_N)\cdot \nabla \tau_*,\qquad
    \oint_{\gamma_N} \tau_N'\, \rmd s = 0.
\end{equation}
The mean-zero condition will be used later on to ensure smallness
of some terms in our iteration; see Lemma \ref{Xpoin}. 

We note that $\tau_N$ has a jump discontinuity across $\Pi$, but
it is smooth away from $\Pi$ if $u_N$ is.
Given this $\tau_N$, we let $T_N = \tau_N|_{\Pi_+} - \tau_N|_{\Pi_-}$ denote the period
of $\gamma_N$, and define $H_N$ by
\begin{align}
  \label{qN}
  H_N &= \mathcal{H}_*(T_N),
\end{align}
and we now define
\begin{align}
  \label{omegaN}
  \omega_N = \nabla \tau_N \times \nabla H_N = \mathcal{H}_*'(T_N) \nabla \tau_N \times \nabla T_N.
\end{align}
Then, as mentioned in the previous section, even though $\tau_N$ jumps across $\Pi$,
$\omega_N$ is continuous
across $\Pi$, and if $\tau_N \in C^{k,\alpha}$ away from $\Pi$ then $\omega \in C^{k-1,\alpha}$
away from $\Pi$ as well. We also note that $\tau_N$ is odd and $H$ is even, so $\omega_N$
is even.

To pass to the next step of the iteration,
we define $u_{N+1}$ by solving the div-curl system
\begin{alignat}{2}
  \label{ellN1}
  \div \,u_{N+1} &= 0, &&\text{ in } D_\ve,\\
  \curl \,u_{N+1} &= \omega_N,&&\text{ in } D_\ve,\\
  u_{N+1}\cdot n &=0, \qquad &&\text{ on } \pa D_\ve.
  \label{ellN3}
\end{alignat}
This does not uniquely determine $u_{N+1}$, because there is a one-dimensional
family $\mathsf{H}$ of harmonic vector fields on $D_\ve$ which are tangent to $\pa D_\ve$.
To get a unique solution $u_{N+1}$, we add the requirement that
\begin{equation}
  \label{proj}
  \mathbb{P}_{\mathsf{H}} u_{N+1} = 0,
\end{equation}
where $\mathbb{P}_{\mathsf{H}}$ denotes the $L^2$-orthogonal projection onto $\mathsf{H}$. We claim the resulting $u_{N+1}$ is odd:
\begin{lemma}
  \label{}
  Define $D_\ve$ as in \eqref{dvedef} and suppose that
    that $\omega_N$ is an even and $m$-fold symmetric vector field on $D_\ve$.
    If $u_{N+1}$ is the unique solution of the system
    \eqref{ellN1}-\eqref{ellN3} satisfying the condition
    \eqref{proj}, then $u_{N+1}$ is odd and $m$-fold symmetric.
\end{lemma}
\begin{proof}
    To cut down on notation, we drop the subscripts $N, N+1$.
    Let $v  = u + R_\Pi^* u$
    where recall $R_\Pi^* u$ denotes the pullback of $u$ by 
    the reflection $R_\Pi$.
    The parity assumptions on $\omega$ and on the domain ensure that
    $v$ is harmonic.

    Letting  $u_{\mathsf{H}}$ be an $L^2-$normalized
    harmonic vector field on $D_\ve$, we can write $\mathbb{P}_{\mathsf{H}} u
    = \int_{D_\ve} u(x) u_{\mathsf{H}}(x)\rmd x$. We claim that under our assumptions,
    $u_{\mathsf{H}}$ is even. Indeed, both its even and odd parts are harmonic
    and are therefore a multiple of $u_{\mathsf{H}}$ (since the space of harmonics is one-dimensional). It follows that
    $u_{\mathsf{H}}$ must be either even or odd. Since $u_{\mathsf{H}}$ is
    a perturbation of the even vector field $u_{\mathsf{H}_0} = e_z$, 
    it follows that it is even.  As a result, $\mathbb{P}_{\mathsf{h}} R_\Pi^* u \rmd x= 
    \int_{D_\ve} R_\Pi^* u u_{\mathsf{H}} = -\int_{D_\ve} u R_\Pi^* \mathsf{u}_{\mathsf{H}} \rmd x
    = -\int_{D_\ve} u  \mathsf{u}_{\mathsf{H}}\rmd x = 0$, by assumption,
    so both $u$ and $R_\Pi^*u$ are orthogonal to $\mathsf{u}_{\mathsf{H}}$. Thus
    $u + R_\Pi^*u = 0,$ as needed. The $m$-fold symmetry follows similarly.
\end{proof}

%
%
%
%

We now show that the sequence $\{u_N\}$ converges.
Becuase the travel time defined in \eqref{tauN} is discontinuous
across $\Pi$, and because each $u_N$ will be
odd, it is convenient to work in the domain
$\widetilde{D}_\ve = D_\ve \cap \{y \geq 0\}$
instead of $D_\ve$. We will show that the restrictions $\widetilde{u}_N = u_N|_{\widetilde{D}_\ve}$
converge with respect to the H\"{o}lder norms
\begin{equation}
  \label{}
  \|u\|_{k, \alpha} = \|u\|_{C^{k, \alpha}(\widetilde{D}_\ve)},
\end{equation}
for $k \geq 2, \alpha \in (0, 1)$, provided $\ve$ is taken sufficiently small.
Letting $\widetilde{u}$ denote the limit of the $\widetilde{u}_N$, we then extend
$\widetilde{u}$ to a vector field $u$ defined in all of $D_\ve$ by parity,
and the resulting vector field will be in $C^{k, \alpha}$ away from
$\Pi$. We will show later on that in fact it is $C^{k,\alpha}$ across
$\Pi$ as well.

We now prove the needed bounds. In order to close our estimates, we will use the following
Poincar\'{e}-type inequality.
\begin{lemma}
  \label{Xpoin}
  Let $U$ be a domain foliated by a family of simple, $C^{k,\alpha}$-smooth
  and $m-$fold symmetric curves; that is, for each curve
  $\gamma$, there is a curve $\gamma_0$ so that
  $\gamma = \bigcup_{j=0}^{m-1}  O_{2\pi/m}^j\gamma_0$.
  Let $X$ be a $C^{k,\alpha}$ vector field on $U$ so that $X|_\gamma$ is tangent
  to $\gamma$ for each curve $\gamma$ in the foliation, and that the $X$ period $T(\gamma) := \int_\gamma \frac{\rmd s}{|X|}\leq C_{\mathsf{per}}$ is bounded uniformly.
  Further assume that $X$ is $m$-fold symmetric in the sense that $X= ({O_{2\pi/m}})_* X$
  where $({O_{2\pi/m}})_* $ denotes the pushforward.  Then, if $f$ is an $m$-fold symmetric function with $\int_{\gamma} f\, \rmd s = 0$ for each
  curve $\gamma$ in the foliation, then there is a constant
  $C = C\left(C_{\mathsf{per}},\|X\|_{C^{k,\alpha}}, \|\gamma\|_{C^{k,\alpha}},
     k, \alpha\right)$ so 

  \begin{equation}
    \label{poin}
    \| f \|_{C^{k,\alpha}(U)} \leq   \frac{C}{m} \| \nabla_X f\|_{C^{k,\alpha}(U)}.
  \end{equation}

\end{lemma}
\begin{proof}

    By the $m$-fold symmetry, we can write $U = \bigcup_{j = 0}^{m-1} (O_{2\pi/m})^j U_0$
    for a domain $U_0$, and it suffices
    to prove the bound \eqref{poin} with $U$ replaced by $U_0$. Fix a curve
    $\gamma$ in the foliation and let $\gamma_0$ denote its restriction to $U_0$. Let $x,y\in \gamma_0$ and let $\Phi_t$ be the flowmap for $X$, i.e.
$$
\frac{\rmd }{\rmd t} \Phi_t = X\circ \Phi_t , \qquad \Phi_0={\rm id}.
$$
Let $T(x,y)$ be such that $\Phi_{T(x,y)}(x) = y$. 
Since $f$ is mean zero on each $\gamma$ and is $m$-fold symmetric,
    $f$ is mean zero on $\gamma_0$, so
integrating the expression $f(x)- f(y) =\int_0^{T(x,y)} (\nabla_X  f)\circ \Phi_t(x) d t$
in $y$ over $\gamma_0$, we find
\begin{equation} \label{finvform}
f(x) = \frac{1}{{\rm length}(\gamma_0)}\int_{\gamma_0}\int_0^{T(x,y)} (\nabla_X  f)\circ \Phi_t(x) \rmd t \rmd y,
\end{equation}
and it follows that for any $x$ on $\gamma_0$,
\begin{align*}
|f(x)| \leq \frac{1}{{\rm length}(\gamma_0)}\int_{\gamma_0}\int_0^{T(x,y)} |(\nabla_X  f)\circ \Phi_t(x)| \rmd t \rmd y 
&\leq  T(\gamma_0)\| \nabla_X f\|_{L^\infty}
\end{align*}
where 
\be
T(\gamma_0):=  \int_{\gamma_0} \frac{\rmd s}{|X|}
\ee
is the time taken to traverse the segment $\gamma_0$, end-to-end. By the $m$-fold
symmetry, $T_*(\gamma_0)
= \frac{1}{m}T_*(\gamma)$. Since the curves $\gamma_0$
foliate the domain $U_0$, this gives the bound when
$k = \alpha = 0$. We now show how to bound the first derivative, higher-order derivatives
being similar.  We express the line integral via its parametrization by the flow of $X$, namely $\Phi_t(x):[s_1(x),s_2(x)] \to \gamma_0$:
\begin{align}
f(x) &= \frac{1}{\int_{s_1(x)}^{s_2(x)} |X\circ \Phi_s(x)|  ds}\int_{s_1(x)}^{s_2(x)} \int_0^{T(x,\Phi_s(x))} (\nabla_X  f)\circ \Phi_t(x) |X\circ \Phi_s(x)| \rmd t \rmd s,
\end{align}
Note that derivatives of $s_i(x)$ and $\Phi_t(x)$ are bounded by derivatives of $X$.  Thus
\begin{align*}
|\nabla f(x)| &\leq  \Bigg[\frac{1}{m} T(\gamma) \|s_i\|_{C^1} \|X\|_{L^\infty} \|\nabla_X  f\|_{L^\infty}  +\frac{\sup_{t\in [s_1,s_2]}\|\nabla_x T(\cdot,\Phi_t (\cdot) )\|_{L^\infty}}{{\rm length}(\gamma_0)} \int_{\gamma_0}|\nabla_X  f(y)| d y \\
&\qquad + \frac{\sup_t\| \nabla \Phi_t(\cdot ) \|_{L^\infty}}{{\rm length}(\gamma_0)}\int_{\gamma_0}\int_0^{T(x,y)} |\nabla \nabla_X  f\circ \Phi_t(x)| \rmd t \rmd y\Bigg].
\end{align*}
Note that, if $f$ is  $m$-fold symmetric, then so is $\nabla_X f$ and, moreover, it is mean zero on each $\gamma_0$.  As such, can can apply the estimate for $\|f\|_{L^\infty}$ to arrive at $\|\nabla_X f\|_{L^\infty}\lesssim 1/m \|\nabla \nabla_X f\|_{L^\infty}.$ Likewise, the last term is bounded by the same argument.  This concludes the sketch of the argument.  Higher derivatives follow similarly.
\end{proof}

We now apply this lemma to $U = \widetilde{D}_\ve$ with $X = u_N$.  The period function of $u_N$ satisfies all assumptions of the above theorem, because $u_N$ is odd, close to $u_*$ and vanishes along the axis. 
Since $\oint_{\gamma_N} \tau_N \rmd s = 0$
for each field line $\gamma_N$, taking $m$ even, it follows from the
$m$-fold symmetry that $\int_{\gamma_N \cap \widetilde{D}_\ve} \tau_N \rmd s = 0$,
and so by \eqref{poin} and the equation \eqref{tauNprime}, $\tau_N'$ satisfies
\begin{equation}
  \label{}
  \|\tau_N'\|_{k,\alpha} \leq \frac{C}{m}  \|u_N'\|_{k, \alpha}\|\nabla \tau_*\|_{k, \alpha},
\end{equation}
for a constant $C > 0$,
where $u_N' = u_N - u_*$ is the perturbation of the velocity field.
Since $T_N = \tau_N|_{\Pi_+} - \tau_N|_{\Pi_-}$, we also have that $T_N' = T_N - T_*$ satisfies
\begin{equation}
  \label{}
  \|T_N'\|_{k, \alpha} \leq \frac{C'}{m}\|u_N'\|_{k, \alpha}\|\nabla \tau_*\|_{k, \alpha},
\end{equation}
for a constant $C' >0$,
and it follows that with $\omega_N$ defined in \eqref{omegaN}, $\omega'_N = \omega_N - \omega_*$
satisfies
\begin{align}
  \label{}
  \|\omega_N'\|_{k, \alpha}
  \lesssim \|T_N\|_{k+1, \alpha}\| \tau_N'\|_{k+1, \alpha} 
  + \|T_N'\|_{k + 1, \alpha}\|\tau_N\|_{k+1,\alpha} 
  \lesssim \frac{1}{m}  \| u_N'\|_{k+1, \alpha},
\end{align}
provided $\|u'_N\|_{k+1, \alpha} \leq 1$, say,
where the implicit constant depends on $k, \alpha, u_*$ and $\mathcal{H}_*$, but not on $m$.
If we now define $u_{N+1}$ by solving \eqref{ellN1}--\eqref{proj}, the perturbed velocity
field $u_{N+1}' = u_{N+1} - u_*$ satisfies the system
\begin{alignat}{2}
  \label{}
  \div u_{N+1}' &= 0, &&\qquad \text{ in } D_\ve,\\
  \curl u_{N+1}' &= \omega_N'&&\qquad \text{ in } D_\ve,\\
  u_{N+1}'\cdot n &= -\ve u_*\cdot \nabla g,&&\qquad \text{ on } \pa D_\ve
\end{alignat}
and so by standard elliptic estimates (see e.g. \cite{ADN64}), it satisfies the bounds
\begin{align}
  \label{}
  \|u_{N+1}'\|_{k+1, \alpha}
  &\lesssim \|\omega_N'\|_{k, \alpha} + 
  \ve\|g\|_{C^{k+1,\alpha}(\mathbb{S}^1\times \mathbb{S}^1)}\\
  &\lesssim  
  \frac{1}{m}\|u_N'\|_{k+1, \alpha} + \ve \|g\|_{C^{k+1,\alpha}(\mathbb{S}^1\times \mathbb{S}^1)}.
\end{align}
It follows that there is a constant $M > 0$ and $m_1, \ve_1$ so that if $m > m_1$ and $\ve < \ve_1$, 
then the bound $ \|u_N'\|_{k+1, \alpha} \leq M$ implies the same bound for $\|u_{N+1}\|_{k+1, \alpha}$.
In the same way, we find
\begin{equation}
  \label{}
  \|u_{N+1}' - u_{N}'\|_{k+1, \alpha} \lesssim \frac{1}{m} 
  \|u_N' - u_{N-1}'\|_{k+1, \alpha},
\end{equation}
for an implicit constant independent of $\ve$ and $m$,
if $\|u_{N}\|_{k+1,\alpha}, \|u_{N+1}'\|_{k+1,\alpha} \leq M.$
It follows that, taking $\ve, 1/m$ smaller if needed, the sequence $\{u_N\}$ converges
in $C^{k+1,\alpha}(\widetilde{D}_\ve)$. The limit $\widetilde{u}$
satisfies \eqref{Eulereqn} in $\widetilde{D}_\ve$, and 
we extend $\widetilde{u}$ to all of $D_\ve$ by
$
  u(p) = (R_\Pi\widetilde{u})(R_\Pi p),
  $
whenever $p \in D_\ve \setminus \widetilde{D}_\ve$, noting that
the reflection $R_\Pi$ maps $D_\ve \setminus \widetilde{D}_\ve$
to $\widetilde{D}_\ve$. 
Then $u$ is a $C^{k+1,\alpha}$ solution of \eqref{Eulereqn} away
from $\Pi$,
and by construction it has continuous vorticity
across $\Pi$. 

We now claim that in fact $u \in C^{k+1, \alpha}$ everywhere.
Indeed, since $\omega \in C^{k,\alpha}$ up to $\Pi$ and since it
is continuous across $\Pi$, it follows that $\omega$ is actually
Lipschitz across $\Pi$, and since $u$ satisfies the div-curl system
\begin{alignat}{2}
  \label{}
  \div u &=0, &&\qquad \text{ in } D_\ve,\\
  \curl u &= \omega, &&\qquad \text{ in } D_\ve\\ 
  u\cdot n &= 0, &&\qquad \text{ on } \pa D_\ve,
\end{alignat}
it follows that $u\in C^{1, \alpha}(D_\ve)$
(in fact, $u\in C^{1, 1-}(D_\ve)$).
Since $u$ satisfies \eqref{Eulereqn} by construction, it
follows that $u$ is a strong solution of the steady Euler equations
\eqref{eul1} and thus that the steady vorticity equation
\begin{equation}
  \label{transport}
    u\cdot \nabla \omega = \omega\cdot\nabla u
\end{equation}
holds. By construction, $\omega$ is $C^{k,\alpha}$ regular in the directions
tangent to $\Pi$ but so far we have only shown that it is continuous across $\Pi$.
Recalling that $u$ is normal to $\Pi$ at $\Pi$, from
\eqref{transport} it follows from the above that $u\cdot \nabla \omega
\in C^{0, \alpha}$ across $\Pi$, and so we have shown
that $\omega \in C^{1, \alpha}$ globally. This then implies $u \in C^{2, \alpha}$,
and repeatedly applying this argument shows that $u \in C^{k, \alpha}$, as needed.
Since the above holds for any $k$, this completes the proof of Thm. \ref{mainthm}.

\vspace{2mm}

\noindent \textbf{Acknowledgements.}  We are grateful to Peter Constantin for his guidance and for sharing with us his vision of fluid dynamics and PDE.  We thank D. Peralta Salas for point out the work of Salat--Kaiser \cite{SK95,KS97}. We thank also the members of the Simons collaboration for Hidden Symmetries and Fusion Energy for discussion and insight on Grad's conjecture.  The work of TDD was supported by the NSF CAREER award \#2235395, a Stony Brook University Trustee’s award as well as an Alfred P. Sloan Fellowship. The work of TME was supported by a Simons Fellowship and the NSF grants DMS-2043024 and DMS-2510472.  The work of DG was supported by a startup grant from Brooklyn College, PSC-CUNY grant TRADB-55-214, and the NSF grant DMS-2406852.

\vspace{-4mm}

\qquad\\
Theodore D. Drivas\\
Department of Mathematics\\
Stony Brook University, Stony Brook NY 11790, USA\\
{\tt tdrivas@math.stonybrook.edu}

\qquad\\
Tarek M. Elgindi\\
Department of Mathematics\\
Mathematics Department, Duke University, Durham, NC 27708, USA\\
{\tt  tarek.elgindi@duke.edu}

\qquad\\
Daniel Ginsberg\\
Department of Mathematics\\
Brooklyn College (CUNY), Brooklyn, NY 11210, USA\\
{\tt daniel.ginsberg@brooklyn.cuny.edu}

\end{document}